\documentclass[11pt,letterpaper,reqno]{amsart}
\usepackage[includehead,includefoot,margin=20mm]{geometry}
\usepackage{comment}
\usepackage{times}
\usepackage{amsfonts} %
\usepackage{amsmath} %
\usepackage{amssymb} %
\usepackage{amsthm} %
\usepackage{enumerate} %
\usepackage[all]{xy}
\usepackage{bbm} %
\usepackage{graphicx} %
\usepackage{nicefrac} %
\usepackage{color} %
\usepackage{hyperref}
\hypersetup{
    colorlinks=true,       
    linkcolor=blue,          
    citecolor=blue,        
    filecolor=blue,      
    urlcolor=blue           
}
\usepackage[nameinlink,capitalize,noabbrev]{cleveref}

\usepackage{calrsfs}

\usepackage{graphicx}

\usepackage[mathscr]{euscript}

\newtheorem{theorem}{Theorem}[section] %
\newtheorem{corollary}[theorem]{Corollary} %
\newtheorem{lemma}[theorem]{Lemma} %
\newtheorem{proposition}[theorem]{Proposition} %
{\theoremstyle{remark} %
  \newtheorem{remark}[theorem]{Remark}} %
{\theoremstyle{definition} %
  \newtheorem{definition}[theorem]{Definition} %
  \newtheorem{example}[theorem]{Example} %
}

\newtheorem{introthm}{Theorem}

\newcommand{\PP}[0]{\ensuremath{\mathbf{P}}}
\newcommand{\CC}[0]{\ensuremath{\mathbf{C}}}
\newcommand{\ZZ}[0]{\ensuremath{\mathbf{Z}}}

\newcommand{\tvarphi}[0]{\ensuremath{\widetilde{\varphi}}}

\newcommand{\Aut}[0]{\ensuremath{\operatorname{Aut}}}

\newcommand{\Lin}[0]{\ensuremath{\operatorname{Lin}}}

\newcommand{\spec}[0]{\ensuremath{\operatorname{Spec}}}
\newcommand{\proj}[0]{\ensuremath{\operatorname{Proj}}}

\newcommand{\GM}[0]{\ensuremath{\mathbf{G}_{\mathrm{m}}}}
\newcommand{\AF}[0]{\ensuremath{\mathbf{A}}}

\newcommand{\QQ}[0]{\ensuremath{\mathbf{Q}}}

\begin{document}

\title{Automorphisms of prime power order of weighted hypersurfaces}

\author[Alvaro Liendo]{Alvaro Liendo (corresponding author)} %
\address{Instituto de Matem\'atica y F\'isica, Universidad de Talca,
  Casilla 721, Talca, Chile} %
\email{aliendo@utalca.cl}

\author{Ana Julisa Palomino}
\address{Instituto de Matem\'atica y F\'isica, Universidad de Talca,
  Casilla 721, Talca, Chile} %
\email{ana.palomino@utalca.cl}

\date{\today}

\thanks{{\it 2020 Mathematics Subject
    Classification}: 14J70, 14J50.\\
  \mbox{\hspace{11pt}}{\it Key words}: Weighted hypersurfaces, automorphisms of weighted hypersurfaces.\\
  \mbox{\hspace{11pt}} Both authors were partially supported by Fondecyt Project 1240101. The second author was also partially supported by CONICYT-PFCHA/Doctorado Nacional/folio 21240560}

\begin{abstract}
We study automorphisms of quasi-smooth hypersurfaces in weighted projective spaces, extending classical results for smooth hypersurfaces in projective space to the weighted setting. We establish effective criteria for when a power of a prime number can occur as the order of an automorphism, and we derive explicit bounds on the possible prime orders. A key role is played by a weighted analogue of the classical Klein hypersurface, which we show realizes the maximal prime order of an automorphism under suitable arithmetic conditions. Our results generalize earlier work by Gonz\'alez-Aguilera and Liendo.
\end{abstract}

\maketitle

\section*{Introduction}

Smooth hypersurfaces in complex projective space are fundamental objects in algebraic geometry, and understanding their automorphism groups is a classical problem. For a smooth hypersurface $X \subset \PP^{n+1}$ of dimension $n$ and degree $d$, it is well known that if $d \ge 3$ and $(n,d) \ne (1,3), (2,4)$, then the automorphism group $\operatorname{Aut}(X)$ is finite, and every automorphism of $X$ is induced by a projective linear transformation of the ambient projective space~\cite{MM64}. Typically, a general hypersurface has a trivial automorphism group, but certain special hypersurfaces exhibit nontrivial automorphism groups. This naturally leads to the following question: \emph{which finite groups can occur as subgroups of $\operatorname{Aut}(X)$?}

A first step towards addressing this question was taken in \cite{GL11}, where all cyclic groups of prime order acting faithfully on smooth cubic hypersurfaces $X \subset \PP^{n+1}$ were classified. This result was soon generalized in \cite{GL13} to classify cyclic groups of prime power order acting faithfully on smooth hypersurfaces of any degree $d \geq 3$. More recently, a complete classification of all abelian groups acting faithfully on smooth hypersurfaces of degree $d \geq 3$ was obtained in \cite{Zhe22}; see also \cite{GALM22}. Full classifications are also available for specific cases, such as smooth cubic threefolds \cite{WY19}, smooth quintic threefolds \cite{OY19}, smooth cubic fivefolds and fourfolds \cite{YYZ24}, and symplectic automorphisms of smooth cubic fourfolds \cite{LZ22}.

Finally, it is worth noting that two recent and independent works \cite{esserlargeautomorphismgroups,yang2025automorphismgroupssmoothhypersurfaces} showed that, with finitely many exceptions, the smooth hypersurface of degree $d\geq 3$ in $\PP^{n+1}$ that admits the largest automorphism group is the Fermat hypersurface, i.e., the zero locus of the homogeneous polynomial
$$
F = x_0^d + x_1^d + \dots + x_n^d + x_{n+1}^d\,.
$$

Another hypersurface of particular interest is the \emph{Klein hypersurface},
defined as the zero locus of the polynomial
$$
K = x_0^{d-1}x_1 + x_1^{d-1}x_2 + \dots + x_n^{d-1}x_{n+1} + x_{n+1}^{d-1}x_0\,.
$$
It was shown in~\cite{GL13} that the Klein hypersurface is, up to isomorphism,
the unique smooth hypersurface of degree~$d$ admitting an automorphism of the
largest possible prime order. One of the main results of the present paper is a
generalization of this fact to the weighted setting; see \cref{th:Klein}.

\medskip

Let $\PP^{n+1}_a$ denote the weighted projective space associated with the weights $a = (a_0, a_1, \ldots, a_{n+1})$, where each $a_i \in \ZZ_{>0}$. A hypersurface $X = V(F) \subset \PP^{n+1}_a$ defined by a homogeneous polynomial $F$ of degree $d$ is called quasi-smooth if its affine cone $\{x\in \AF^{n+2}\mid F(x)=0\}$ is smooth away from the origin; see~\cite[p.~35]{Dol81} or~\cite[Section~2]{BaCo94}. 

In this paper, we build upon the classical results of González-Aguilera and Liendo ~\cite{GL11,GL13} and their subsequent developments ~\cite{GALM22,GLMV24} by transferring the study of automorphisms of smooth projective hypersurfaces to the broader setting of quasi-smooth hypersurfaces in weighted projective spaces. While several arguments parallel those in the unweighted case, the weighted context introduces substantial new phenomena: automorphisms of the ambient space need not be linear, quasi-smoothness replaces smoothness, and the interplay between the degree $d$ and the system of weights $a=(a_0,a_1,\dots,a_{n+1})$ requires a refined arithmetic analysis.

The automorphism group $\operatorname{Aut}(X)$ of a quasi-smooth hypersurface in weighted projective space is not necessarily induced by automorphisms of the ambient space. However, under mild conditions analogous to those in the classical setting~\cite{MM64}, it coincides with the group of automorphisms induced by $\operatorname{Aut}(\PP^{n+1}_a)$~\cite[Theorem~2.1]{esser2024automorphisms}. The group $\operatorname{Aut}(\PP^{n+1}_a)$ itself admits a natural description in terms of the group of equivariant automorphisms of the affine space $\AF^{n+2}$ modulo the image of the diagonal $\GM$-action defined by the weights $a$, where $\GM$ stands for the multiplicative group of the base field $\CC$~\cite{AAm89}.

Let $\PP^{n+1}_a$ be the weighted projective space with weights $a = (a_0, a_1, \ldots, a_{n+1})$, where each $a_i \in \ZZ_{>0}$, and let $d \geq 3$. Let also $p,r$ be positive integers with $p$ prime. The main technical results of this paper include a necessary condition for a cyclic group of order $p^r$ to act faithfully on a quasi-smooth hypersurface of degree $d$ in $\PP^{n+1}_a$, see \cref{order_automorphism_weight}; as well as a sufficient condition for such an action under additional hypotheses, see \cref{order_automorphism}. 

 In the special case where each weight $a_i$ divides $d$, we establish a complete
criterion for the prime numbers that can occur as the order of an automorphism
of a quasi-smooth hypersurface of degree $d$ in $\PP^{n+1}_a$,
generalizing~\cite[Proposition~2.2]{GL13}.

\begin{introthm}[See \cref{order_weights_divides_d}]
Let $n, d, p$ be positive integers, with $p$ prime and $d\geq 3$. Let
$a \in \ZZ_{>0}^{n+2}$ be such that $a_i$ divides $d$ for all
$i\in\{0,1,\ldots,n+1\}$. Assume further that either $n\geq 3$, or $n=2$
and $a_0+a_1+a_2+a_3\neq d$. Then $p$ is the order of an automorphism of a
well-formed quasi-smooth hypersurface of $\PP_a^{n+1}$ if and only if one of
the following holds:
\begin{enumerate}
  \item[$(a)$] $p$ divides $d$; or
  \item[$(b)$] $a_ip$ divides $d-a_j$, for some $i,j\in\{0,1,\ldots,n+1\}$
    with $i\neq j$; or
  \item[$(c)$] after possibly reordering the weights, there exists
    $\ell$ with $1\leq\ell\leq n+1$ such that $a_0=a_1=\cdots=a_\ell$ and
    $$\left(1-\frac{d}{a_0}\right)^{\ell+1}\equiv 1\pmod{p}\,.$$
\end{enumerate}
\end{introthm}
We apply this result to derive an explicit bound on the prime orders that can occur, see \cref{cota}. Moreover, the hypothesis that each $a_i$ divides $d$ arises naturally in several geometric contexts, such as the study of $K$-stability of weighted Fano hypersurfaces~\cite{Sano_Tasin} and the classification of Fano threefolds containing a smooth rational surface with ample normal bundle~\cite{Campana_Flenner}.

We then turn our attention to quasi-smooth hypersurfaces in weighted projective spaces in the complementary case where each weight $a_i$ is relatively prime to $d$. Under this assumption, we establish the following bound: if a cyclic group of prime order $p$ acts faithfully on a quasi-smooth hypersurface of degree $d$ in $\PP^{n+1}_a$, then
$$
p < \left(\frac{\max(a)}{d - \max(a)}\right) \prod_{t=0}^{n+1} \left( \frac{d - a_t}{a_t} \right)\,,
$$
where $\max(a)$ denotes the maximum of the weights, see \cref{cota-klein}. This result generalizes \cite[Corollary~2.4]{GL13}. Let now
$$p=\frac{1}{d}\left\{\prod_{t=0}^{n+1}\left(\frac{d-a_t}{a_t}\right)+(-1)^{n+1}\right\}$$
and assume that $p$ is a prime number. We note that the quantity in braces 
is always an integer divisible by $d$ when $\gcd(a_i,d)=1$ for all $i$; see 
\cref{p_maximum}. Moreover, \cref{p_maximum} shows that $p$ is the largest 
prime that can be the order of an automorphism of a quasi-smooth hypersurface 
of degree $d$ in $\PP^{n+1}_a$. In \cref{th:Klein}, we prove that any 
quasi-smooth hypersurface admitting an automorphism of this maximal prime 
order $p$ is isomorphic to a weighted Klein hypersurface.

Weighted Klein hypersurfaces were also studied in the final section of \cite{kollar}, where Kollár used them to construct examples of rational $\QQ$-homology projective planes with ample canonical class. More recently, in \cite{urzua}, Urzúa and Yáñez focused on the surface case in detail, referring to these as Kollár hypersurfaces. They provided a complete classification of their minimal models and described these surfaces as cyclic covers of the plane branched along line arrangements.
 
\medskip

The paper is organized as follows. In \cref{sec1} we review basic definitions concerning weighted projective spaces, quasi-smoothness, and automorphism groups of weighted quasi-smooth hypersurfaces. In \cref{sec2}, we establish our main result: a criterion for the existence of automorphisms of prime power order in the quasi-smooth weighted setting. Finally, in \cref{sec3} we investigate the structure of automorphisms of maximal possible prime order and their relation to weighted Klein hypersurfaces.

\section{Automorphism groups of weighted projective spaces and quasi-smooth hypersurfaces} \label{sec1}

Letting $n \geq 1$, we fix $a = (a_0, a_1, \ldots, a_{n+1}) \in \ZZ_{>0}^{n+2}$ such that $\gcd(a_0, a_1, \dots, a_{n+1}) = 1$. Let $\AF^{n+2} = \spec \CC[x_0, \ldots, x_{n+1}]$, and consider the $\GM$-action
$$
\alpha \colon \GM \times \AF^{n+2} \to \AF^{n+2} \quad \text{given by} \quad (t, (x_0, \dots, x_{n+1})) \mapsto (t^{a_0}x_0, \dots, t^{a_{n+1}}x_{n+1})\,.
$$
The condition $\gcd(a_0, a_1, \dots, a_{n+1}) = 1$ is equivalent to the action $\alpha$ being faithful. This $\GM$-action induces a grading on $\CC[x_0, \ldots, x_{n+1}]$ where each variable $x_i$ has degree $a_i$, for all $i \in \{0, 1, \dots, n+1\}$. We define the weighted projective space with weights $a = (a_0, a_1, \ldots, a_{n+1})$, denoted by $\PP^{n+1}_a$, as $\proj \CC[x_0, \ldots, x_{n+1}]$, see \cite[Ch.~II, Proposition~2.5]{Har77} for details of the $\proj$ construction.

In the sequel, we always consider the polynomial ring $\CC[x_0,\dots,x_{n+1}]$ with the grading given by $a\in \ZZ_{>0}^{n+2}$. The usual projective space is recovered from this construction by setting each $a_i=1$. Remark that the vector $a$ is not uniquely determined by $\PP_a^{n+1}$ even up to reordering. For example, if $a=(1,2,\dots,2)$, then $\PP_a^{n+2}$ is isomorphic to the usual projective space $\PP^{n+2}$. The standard references for weighted projective spaces are \cite{Dol81,Ian00}.

The automorphism groups of weighted projective spaces are known and are a natural generalization of the case of the usual projective space, see \cite{AAm89,esser2024automorphisms}. Recall that $\Aut_{\GM}(\AF^{n+2})$ is the group of $\GM$-equivariant automorphisms of $\AF^{n+2}$. Then 
\begin{align} \label{quotient-aut-wps}
\Aut(\PP_{a}^{n+1})=\Aut_{\GM}(\AF^{n+2})/H\,,
\end{align}
where $H$ stands for the image of the $\GM$-action $\alpha$ inside $\Aut_{\GM}(\AF^{n+2})$. Remark that if $a=(1,1,\ldots,1)$, then being $\GM$-equivariant stand for sending each variable to an element of degree 1, so $\Aut_{\GM}(\AF^{n+2})=\operatorname{GL}(n+2,\CC)$ and $\Aut(\PP_{a}^{n+1})=\operatorname{PGL}(n+2,\CC)$ in this case. Furthermore, we have the following generalization of the fact that finite abelian subgroups of $\operatorname{GL}(n+2,\CC)$ are diagonalizable.
\begin{lemma}[{\cite[Lemma~1.4]{esser2024automorphisms}}]\label{L1.1}
    Let $G$ be a finite abelian subgroup of $\Aut_{\GM}(\AF^{n+2})$. Then $G$ is conjugated to a diagonal automorphism of $\Aut_{\GM}(\AF^{n+2})$, i.e., an automorphism sending each $x_i$ to a scalar multiple of itself. 
\end{lemma}

Let now $a = (a_0, a_1, \dots, a_{n+1}) \in \ZZ_{>0}^{n+2}$ be such that $\gcd(a_0, a_1, \dots, a_{n+1}) = 1$. A polynomial 
$$
F \in \CC[x_0, \dots, x_{n+1}]\,,
$$ 
is called homogeneous if it is homogeneous with respect to the grading induced in $\CC[x_0, \ldots, x_{n+1}]$ by $a$. Let $F \in \CC[x_0, \ldots, x_{n+1}]$ be an irreducible homogeneous polynomial. A hypersurface in $\PP^{n+1}_a$ is the algebraic variety $X = V(F)$ defined as the zero locus of $F$ in the weighted projective space $\PP^{n+1}_a$, that is,
$$
X = \left\{ [x_0 : \dots : x_{n+1}] \in \PP^{n+1}_a \mid F(x) = 0 \right\}\,.
$$

The weighted projective space $\PP^{n+1}_a$ is said to be well-formed if the corresponding $\GM$-action $\alpha$ on $\AF^{n+2}$ has trivial stabilizers in codimension one. This occurs if and only if the greatest common divisor of every subset of $\{a_0, a_1, \dots, a_{n+1}\}$ of size $n+1$ is equal to $1$. By \cite[Lemma~5.7]{Ian00}, every weighted projective space is isomorphic to a well-formed one. Furthermore, a hypersurface $X \subset \PP^{n+1}_a$ is said to be well-formed if $\PP^{n+1}_a$ is well-formed and the intersection of $X$ with the singular locus of $\PP^{n+1}_a$ has codimension at least two in $X$.

Weighted projective spaces $\PP^{n+1}_a$ that are not isomorphic to the usual projective space $\PP^{n+1}$ are always singular. For this reason, a weaker notion of smoothness, known as quasi-smoothness, is introduced~\cite[p.~35]{Dol81}.. Let $F \in \CC[x_0, \dots, x_{n+1}]$ be a homogeneous irreducible polynomial of degree $d$ defining a hypersurface $X = V(F) \subset \PP^{n+1}_a$. The affine cone $C_X$ over $X$ is the affine variety
$$
C_X = \left\{ (x_0, \ldots, x_{n+1}) \in \AF^{n+2} \mid F(x) = 0 \right\}\,.
$$
We say that $X$ is quasi-smooth if $C_X \setminus \{\overline{0}\}$ is smooth. Note that the notion of quasi-smoothness depends on the specific embedding $X \subset \PP^{n+1}_a$. Quasi-smooth hypersurfaces of degree $d$ in $\PP^{n+1}_a$ exist only for certain combinations of weights and degree. Given a fixed choice of weights $a \in \ZZ_{>0}^{n+2}$, a characterization of the degrees $d$ for which quasi-smooth hypersurfaces exist is provided in~\cite[Theorem~8.1]{Ian00}. We present here the formulation given in~\cite{esser2024automorphisms}.

\begin{theorem}[{\cite[Theorem~8.1]{Ian00}}] \label{Iano}\label{generic-smoothness}
Let $a\in \ZZ_{>0}^{n+2}$ and assume that $\PP_a^{n+1}$ is well-formed. Then, there exists a quasi-smooth hypersurface $X$ of degree $d$ in the weighted projective space $\PP_a^{n+1}$ if and only if one of the following conditions holds:
\begin{enumerate}
    \item[$(i)$] $a_i = d$ for some $i \in \{0,1, \ldots, n+1\}$, or
    \item[$(ii)$] for each nonempty subset $I$ of $\{0,1, \ldots, n+1\}$, either
    \begin{enumerate}
        \item $d$ is contained in the subsemigroup of $\ZZ_{\geq 0}$ generated by the weights $\{a_i\mid i \in I\}$, or
        \item there are at least $|I|$ indices $j \notin I$ such that $d - a_j$ is contained in the subsemigroup of $\ZZ_{\geq 0}$ generated by the weights $\{a_i\mid i \in I\}$.
    \end{enumerate}
\end{enumerate}
Moreover, if $(i)$ or $(ii)$ holds, then the general hypersurface of degree $d$ is quasi-smooth.
\end{theorem}

The following lemma is a version of \cref{Iano} applied to singleton sets $I=\{i\}$. A direct proof is straightforward, see also \cite[Lemma~1.2]{GL13}.

\begin{lemma}[{\cite[Proposition~1.2]{esser2024automorphisms}}]\label{quasi-smooth_hypersurface}
Let $X\subset\PP_{a}^{n+1}$ be a quasi-smooth hypersurface that is not a 
linear cone, given by a polynomial $F\in \CC[x_0,\ldots,x_{n+1}]$ of degree 
$d$. Then for each $i\in\{0,1,\ldots,n+1\}$, there exists a monomial of 
degree $d$ with nonzero coefficient in $F$ having the form either $x_{i}^{k}$ 
or $x_{i}^{k}x_j$, for some $j\not=i$.
\end{lemma}

Let $X = V(F)$ be a hypersurface in $\PP^{n+1}_a$, defined as the zero locus of a homogeneous polynomial $F \in \CC[x_0, \ldots, x_{n+1}]$ of degree $d$. The group of linear automorphisms, denoted by $\Lin(X)$, is the subgroup of $\Aut(X)$ consisting of automorphisms that extend to automorphisms of the ambient space $\PP^{n+1}_a$, that is,
$$
\Lin(X) = \{ \varphi \in \Aut(\PP^{n+1}_a) \mid \varphi(X) = X \}\,.
$$
We refer to the elements of $\Lin(X)$ as linear automorphisms, in analogy with the classical case, even though $\Aut(\PP^{n+1}_a)$ is not necessarily the image of linear automorphisms of $\AF^{n+2}$ under the quotient described in~\eqref{quotient-aut-wps}.

We say that $X$ is a linear cone if $d = a_i$ for some $i \in \{0, 1, \ldots, n+1\}$. In this case, the variable $x_i$ appears as a linear summand in $F$, and thus $X$ is isomorphic to the weighted projective space $\PP^n_{a'}$, where $a' = (a_0, a_1, \ldots, a_{i-1}, a_{i+1}, \ldots, a_{n+1})$.

It is a classical result of Matsumura and Monsky~\cite{MM64} that for a hypersurface of degree $d\geq 3$ in the usual projective space $\PP^{n+1}$ (i.e., with weights $a = (1,1,\ldots,1)$), we have $\Aut(X) = \Lin(X)$ and the group $\Aut(X)$ is finite, except in the two exceptional cases $(n,d) = (1,3)$ and $(n,d) = (2,4)$. A natural generalization of this result to the setting of weighted projective spaces has been established in~\cite[Theorem~2.1]{esser2024automorphisms}. We now recall this result.

\begin{theorem} \label{weighted-matsumura-monsk} %
  Let $X\subset\PP_a^{n+1}$ and $X'\subset\PP_{a'}^{n+1}$ be well-formed and quasi-smooth hypersurfaces of degrees $d$ and $d'$, respectively. Assume further that $X$ is not a linear cone and let $\tau\colon X\to X'$ be an isomorphism.

  \begin{enumerate}
      \item[$(i)$]  If $n\geq 3$; or $n=2$ and $a_0+a_1+a_2+a_3\neq d$, then $d=d'$, $a=a'$ up to reordering and $\tau$ is the restriction of an automorphism of $\PP_a^{n+1}$. In particular, $\Aut(X)=\Lin(X)$. 
      \item[$(ii)$] The group $\Lin(X)$ is finite if and only if $d>2\max(a)$, or $d=2\max(a)$ and the maximum is achieved only in one $a_i$. 
  \end{enumerate}
\end{theorem}

Unless otherwise stated, we assume throughout the paper that all hypersurfaces in weighted projective space are well-formed and not linear cones.

\section{Automorphisms of prime power order of quasi-smooth hypersurfaces} \label{sec2}

In this section, we compute the powers of prime numbers that can occur as the order of an automorphism of a well-formed quasi-smooth hypersurface in a weighted projective space. We begin by introducing the notion of $F$-liftability. Let $G$ be a subgroup of $\Aut(\PP^{n+1}_a)$. We say that $G$ is liftable if there exists a subgroup $\widetilde{G} \subset \Aut_{\GM}(\AF^{n+2})$ such that the canonical homomorphism 
$$
\pi\colon \Aut_{\GM}(\AF^{n+2}) \to \Aut(\PP^{n+1}_a)
$$
from~\eqref{quotient-aut-wps} restricts to an isomorphism $\pi|_{\widetilde{G}} \colon \widetilde{G} \to G$. Given an element $\varphi \in G$, we say that $\widetilde{\varphi} \in \widetilde{G}$ is a lifting of $\varphi$ if $\pi(\widetilde{\varphi}) = \varphi$.

Let now $G$ be a subgroup of $\Lin(X)$, where $X$ is a quasi-smooth hypersurface in $\PP^{n+1}_a$ defined by a homogeneous polynomial $F \in \CC[x_0, \ldots, x_{n+1}]$ of degree $d$. We say that $G$ is $F$-liftable if $G$ is liftable and, for every $\widetilde{\varphi} \in \widetilde{G}$, we have $\widetilde{\varphi}^*(F) = F$~\cite{OY19,GALM22}. Moreover, we say that an element $\varphi \in G$ is $F$-liftable if the cyclic group generated by $\varphi$ is $F$-liftable.

We will use the following multi-index notation. Let $F$ be a homogeneous polynomial in $\CC[x_0, \ldots, x_{n+1}]$ of degree $d$. Then we can write
$$
F = \sum_i \lambda_i \cdot x^{m_i}\,,
$$
where each $\lambda_i \in \CC$, and $m_i = (m_{0,i}, m_{1,i}, \dots, m_{n+1,i}) \in \ZZ_{\geq 0}^{n+2}$ is a multi-index. Here, $x^{m_i}$ denotes the monomial $x_0^{m_{0,i}} x_1^{m_{1,i}} \cdots x_{n+1}^{m_{n+1,i}}$. The condition that $F$ is homogeneous of degree $d$ with respect to the grading induced by $a = (a_0,a_1,\ldots, a_{n+1})$ means that
$$
\sum_{j=0}^{n+1} a_j m_{j,i} = d \quad \text{for each } i\,.
$$
With these definitions and notations, we now state the following lemma.

\begin{lemma} \label{lem:liftability}
Let $n, d, p, r$ be positive integers, with $p$ prime. Let $a \in \ZZ_{>0}^{n+2}$, and assume that $p$ does not divide $d$. Further, assume that the conditions of \cref{weighted-matsumura-monsk}~$(i)$ hold. If $\varphi$ is an automorphism of order $p^r$ of a quasi-smooth hypersurface $X \subset \PP^{n+1}_a$, then $\varphi$ is $F$-liftable.
\end{lemma}

\begin{proof}
Assume that $X$ is defined by a polynomial $F \in \CC[x_0, \ldots, x_{n+1}]$ of degree $d$ with respect to the grading given by $a = (a_0,a_1, \ldots, a_{n+1})$. Let $\widetilde{\varphi} \in \Aut_{\GM}(\AF^{n+2})$ be a lifting of an automorphism $\varphi \in \Aut(\PP^{n+1}_a)$. By \cref{L1.1}, we may and will assume that $\widetilde{\varphi}$ is diagonal. Then,
$$
\widetilde{\varphi}^*(F) = \xi^h F \,,
$$
where $\xi$ is a primitive $p^r$-th root of unity and $h \in \ZZ$. Since $p^r$ and $d$ are relatively prime, we can choose $k \in \ZZ$ such that $kd \equiv -h \pmod{p^r}$. Now, consider the automorphism
$$
\widetilde{\psi} \in \Aut_{\GM}(\AF^{n+2}) \quad \text{defined by} \quad \widetilde{\psi}^*(x_i) = \xi^{k a_i} \widetilde{\varphi}^*(x_i)\,.
$$

Now, assuming that $F=\sum_{i}\lambda_i\cdot x^{m_i}$ we obtain that
    \begin{align*}
        \widetilde{\psi}^{*}(F)&= \sum_i \lambda_i\cdot \widetilde{\psi}^*(x^{m_i}) \\ 
        &=\sum_{i} \lambda_i\cdot(\xi^{ka_0}\tvarphi^*(x_{0}))^{m_{0,i}}\cdots (\xi^{ka_{n+1}}\tvarphi^*(x_{n+1}))^{m_{n+1,i}}\\
        &=\xi^{kd}\sum_{i}\lambda_i\cdot \tvarphi^*(x_{0})^{m_{0,i}}\cdots \tvarphi^*(x_{n+1})^{m_{n+1,i}} \\
        &=\xi^{kd} \tvarphi^*(F)=\xi^{kd+h} F=F
    \end{align*}
    
    Note that $\widetilde{\psi}$ and $\widetilde{\varphi}$ induce the same automorphism in $\PP^{n+1}_{a}$ by \eqref{quotient-aut-wps}. We conclude that $\varphi$ is $F$-liftable.
\end{proof}

We now establish our main technical result, which generalizes~\cite[Theorem~1.3]{GL13}.

\begin{proposition}\label{order_automorphism_weight}
Let $n, d, p, r$ be positive integers, with $p$ prime and $d\geq 3$. Let $a \in \ZZ_{>0}^{n+2}$, and assume that $p$ does not divides $d$ and $a_ip$ does not divide $d-a_j$, for $i\ne j$. Further, assume that the conditions of \cref{weighted-matsumura-monsk}~$(i)$ hold. If $p^r$ is the order of an automorphism of a quasi-smooth hypersurface in $\PP^{n+1}_a$, then, after possibly reordering the weights $a = (a_0, a_1, \ldots, a_{n+1})$, there exists $\ell \in \{1,2,\dots,n+1\}$ such that the following conditions hold:
\begin{enumerate}
    \item[$(i)$] $d \equiv a_0 \pmod{a_{\ell}}$, and $d \equiv a_{i+1} \pmod{a_i}$ for all $0 \leq i < \ell$;
    \item[$(ii)$] $\displaystyle{(-1)^{\ell+1} \prod_{t=0}^{\ell} \left( \frac{d - a_t}{a_t} \right) \equiv 1 \pmod{p^r}}$.
\end{enumerate}
\end{proposition}

\begin{proof}
Let $X$ be a quasi-smooth hypersurface in $\PP^{n+1}_a$, defined by a polynomial $F \in \CC[x_0, \ldots, x_{n+1}]$ of degree $d$ with respect to the grading given by $a = (a_0,a_1, \ldots, a_{n+1})$. To prove the proposition, assume that $X$ admits an automorphism $\varphi$ of order $p^r$. By \cref{L1.1}, we may and will assume that $\varphi$ is diagonal. Moreover, by \cref{lem:liftability}, we may also assume that $\varphi$ is $F$-liftable. Let $\widetilde{\varphi}$ be a lifting of $\varphi$ such that $\widetilde{\varphi}^*(F) = F$. 

Letting $\xi$ be a primitive $p^r$-th root of unity, we have
$$
\widetilde{\varphi}(x_0, \ldots, x_{n+1}) = (\xi^{\sigma_0} x_0, \ldots, \xi^{\sigma_{n+1}} x_{n+1}) \quad \text{with} \quad 0 \leq \sigma_i < p^r\,.
$$

 By \cref{quasi-smooth_hypersurface}, the homogeneous polynomial $F$ contains 
a monomial of the form $x_i^{m_i}$ or $x_i^{m_i}x_j$ for each 
$i \in \{0,1,\ldots,n+1\}$ and some $j \ne i$. Since $\varphi$ has order 
$p^r$, there exists $i\in\{0,1,\ldots,n+1\}$ such that $\sigma_i\not=0$. 
For such $i$, assume that $F$ contains a monomial of the form $x_i^{m_i}$. 
Then
$$
a_i m_i = d \quad \text{and} \quad \sigma_i m_i \equiv 0 \pmod{p^r}\,.
$$
This implies that $p^r$ divides $\sigma_i m_i$. Since $\sigma_i \not\equiv 0 
\pmod{p}$, we conclude that $p$ divides $m_i$, and hence $p$ divides $d$, 
contradicting our assumption. Hence, for that $i$, the polynomial $F$ must 
contain a monomial of the form $x_i^{m_i}x_j$ for some $j \ne i$. Up to reordering the weights $a = (a_0, a_1, \ldots, a_{n+1})$, we may assume that $\sigma_0 \ne 0$ and the monomial $x_{{0}}^{m_{0}}x_{1}\in F$, which is invariant by the automorphism $\widetilde{\varphi}^{*}$, so $$a_0m_0+a_1=d \quad\text{and} \quad \sigma_0 m_{0}+\sigma_{1}\equiv 0\pmod{p^r},$$ then 
\begin{align}\label{e4}
   \sigma_{1}\equiv -\left(\frac{d-a_{{1}}}{a_{0}}\right)\sigma_0\pmod{p^r}\,,
\end{align}
and since $a_0p$ does not divide $d-a_1$, we conclude that $\sigma_1\not\equiv 0\pmod{p^r}$. Hence, the monomials
\begin{align*}
x_0^{m_0}x_1,\ x_1^{m_1}x_2,\ \ldots,\ x_\ell^{m_\ell}x_0
\end{align*}
belong to $F$, for some $1 \leq \ell \leq n+1$. Since each monomial has degree $d$, we obtain the following equalities:
\begin{align}\label{e2}
a_0 + m_\ell a_\ell = d \quad \text{and} \quad a_i m_i + a_{i+1} = d \quad \text{for all } 0 \leq i < \ell\,.
\end{align}
These relations yield condition~$(i)$ in the proposition.

Since the monomial $x_{1}^{m_{1}}x_{2}\in F$ is invariant by $\tvarphi^*$, we obtain that $\sigma_{1}m_{1}+\sigma_{{2}}\equiv 0\pmod{p^r}$ and by \eqref{e2} and \eqref{e4}, we have that $\sigma_2\not\equiv 0\pmod{p^r}$ and
\begin{align*}
    \sigma_{2}\equiv (-1)^{2}\ \frac{(d-a_{{2}})(d-a_{{1}})}{a_{1}a_{0}}\ \sigma_0\pmod{p^r}\,.
\end{align*}
Iterating this process for all $j\in\{2,3,\ldots,\ell-1\}$, we obtain that 
\begin{align}\label{e6}
   \sigma_{{j+1}}\equiv (-1)^{j+1}\ \frac{\prod_{t=1}^{j+1}(d-a_{{t}})}{\prod_{t=0}^{j}a_{t}}\ \sigma_0 \equiv (-1)^{j+1}\prod_{t=1}^{j+1}\left(\frac{d-a_t}{a_{t-1}}\right)\sigma_0\pmod{p^r}\,.
\end{align} 
Finally, the monomial $x_{\ell}^{m_{\ell}}x_{0}\in F$ is also invariant by the automorphism $\widetilde{\varphi}^{*}$, so $\sigma_{\ell}m_{\ell}+\sigma_0\equiv 0\pmod{p^r}$. By \eqref{e2} and \eqref{e6}, we conclude that
 \begin{align*}
     1\equiv (-1)^{\ell+1} \prod_{t=0}^{\ell}\left(\frac{d-a_{{t}}}{a_{t}}\right)\pmod{p^r} \,.
 \end{align*}
This proves $(ii)$ and concludes the proof.
 \end{proof}

Let $\varphi$ be an automorphism of order $p^r$ of a quasi-smooth hypersurface $X \subset \PP^{n+1}_a$. As in the proof of \cref{order_automorphism_weight}, we may and will assume that $\varphi$ is induced by a diagonal automorphism of the affine space $\AF^{n+2}$. Hence,
$$
\widetilde{\varphi} \colon \AF^{n+2} \to \AF^{n+2} \quad \text{is given by} \quad (x_0, \ldots, x_{n+1}) \mapsto (\xi^{\sigma_0} x_0, \ldots, \xi^{\sigma_{n+1}} x_{n+1})\,,
$$
where $\xi$ is a primitive $p^r$-th root of unity and $\sigma_i \in \ZZ$. Since $\xi^{p^r} = 1$, the integers $\sigma_i$ may be considered modulo $p^r$, i.e., as elements of $\ZZ/p^r\ZZ$. The vector $(\sigma_0, \sigma_1, \ldots, \sigma_{n+1}) \in (\ZZ/p^r\ZZ)^{n+2}$ is called the signature of $\varphi$.

\begin{remark}\label{liftable}
Let $\varphi$ be an automorphism of order $p^r$ of a quasi-smooth hypersurface $X = V(F)$, with $\gcd(p, d) = 1$ and $\gcd(a_ip, d - a_j) = 1$ for some $i \in \{0,1,\dots,n+1\}$ and $j\ne i$. Let $\sigma_0$ be as in the proof of \cref{order_automorphism_weight}. By \eqref{e6}, $\gcd(p, \sigma_0) = 1$, since otherwise the order of $\varphi$ is less than $p^r$. Hence, there exists $k \in \{1, 2, \dots, p - 1\}$ such that $k \sigma_0 \equiv 1 \pmod{p^r}$. Hence, we can replace $\varphi$ by $\varphi^k$, which generates the same cyclic subgroup of $\Aut(X)$. Therefore, we may and will assume that $\sigma_0 = 1$.

With this assumption, it follows from the proof of \cref{order_automorphism_weight} that the automorphism $\varphi$ has signature
\begin{align*}
\sigma = \left(1,\ 
- \frac{d - a_1}{a_0},\ 
(-1)^2 \frac{(d - a_2)(d - a_1)}{a_1 a_0},\ 
\ldots,\ 
(-1)^{\ell} \prod_{t = 1}^{\ell} \left( \frac{d - a_t}{a_{t-1}} \right),\ 
\underbrace{*, \ldots, *}_{(n+2 - \ell)\text{-times}} 
\right) \in (\ZZ/p^r\ZZ)^{n+2}\,,
\end{align*}
where the entries denoted by $*$ indicate unknown values.
\end{remark}

\begin{example}
In this example we show that the converse of \cref{order_automorphism_weight} 
does not hold. We prove that there is no quasi-smooth hypersurface in $\PP^4_a$ 
with $a=(3,7,2,4,5)$ admitting an automorphism of order $p=23$. Note that 
since $\gcd(23,37)=1$ and $\gcd(a_i\cdot 23, 37-a_j)=1$ for all $i\neq j$, 
by the proof of \cref{order_automorphism_weight} and \cref{liftable}, any 
such automorphism must have the form described in \cref{liftable} with 
$\sigma_0=1$. It therefore suffices to rule out automorphisms of that form. Lets consider the weights $a=(3,7,2,4,5)$ and let $\PP^{n+1}_a=\PP^4_a$ be the corresponding weighted projective space. We also fix the degree $d=37$. We have that the conditions $(i)$ and $(ii)$ in \cref{order_automorphism_weight} hold with $p=23$ and $r=1$. Indeed, taking $\ell=2$ yields $(i)$ and $(ii)$ follows since
\begin{align*}
        (-1)^3 \left(\frac{37-3}{3}\right)\left(\frac{37-7}{7}\right)\left(\frac{37-2}{2}\right)=-23\cdot 37+1\equiv 1\pmod{23}.
\end{align*}
Nevertheless, we will prove that there is no quasi-smooth hypersurface in the weighted projective space with weight $a=(3,7,2,4,5)$ that admits an automorphism $\varphi$ of order $p=23$. 

Assume such $\varphi$ exists and let $\tvarphi$ be a lifting that leaves $F$ invariant. Such lifting exists by \cref{lem:liftability}. By \cref{liftable}, the polynomial $F\in \CC[x_0,\dots,x_{n+1}]$ defining a quasi-smooth hypersurface in the weighted projective space $\PP^4_a$  need to contain the monomials $x_0^{10}x_1$, $x_1^{5}x_2$, and $x_2^{17}x_0$  with non-zero coefficient and the signature of the automorphism $\varphi$ is
$$\sigma = (\sigma_0,\sigma_1,\sigma_2,\sigma_3,\sigma_4) = (1,-10,10\cdot 5,*,*)=(1,13,4,*,*)\in (\ZZ/23\ZZ)^{5}\,.$$
By \cref{quasi-smooth_hypersurface}, we have that $x_3^{8}x_4$ and at least one of the monomials $x_4^{6}x_1$ and $x_4^{7}x_2$ have to appear in $F$ with non-zero coefficient. A straightforward application of the Jacobian Criterion, shows that indeed $x_4^{6}x_1$ and $x_4^{7}x_2$ both have to appear with non-zero coefficient. Finally, these two monomial provide contradictory conditions for the weight of $x_4$. Indeed, the fact that $F$ is invariant by $\tvarphi^*$ yields
$$6\sigma_4+\sigma_1\equiv 0\pmod{23} \quad\mbox{and}\quad 7\sigma_4+\sigma_2\equiv 0\pmod{23}\,.$$
This provides a contradiction since the first equation yields $\sigma_4=17$ while the second one yields $\sigma_4=6$.

\end{example}

In the next proposition, we provide a partial converse to \cref{order_automorphism_weight}.

\begin{proposition}\label{order_automorphism}
    Let $n, d, p, r$ be positive integers, with $p$ prime and $d\geq 3$. Let $a \in \ZZ_{>0}^{n+2}$ and assume that $p$ does not divides $d$ and $a_ip$ does not divide $d-a_j$, for $i\ne j$. Further, assume that the conditions of \cref{weighted-matsumura-monsk}~$(i)$ hold. If, after possibly reordering the weights $a = (a_0,a_1, \ldots, a_{n+1})$, the following conditions are satisfied:
  \begin{enumerate}
      \item[$(i)$] $d \equiv a_0 \pmod{a_{\ell}}$, and $d \equiv a_{i+1} \pmod{a_i}$ for all $0 \leq i < \ell$; 
    \item[$(ii)$]$\displaystyle{(-1)^{\ell+1}\prod_{t=0}^{\ell} \left(\frac{d-a_t}{a_t}\right) \equiv 1 \pmod{p^r}}$; and
      \item[$(iii)$] There exists a quasi-smooth hypersurface $X=V(F)$ of degree $d$ in $\mathbf{P}_a^{n+1}$ with $F=F_1+F_2$ where $F_1\in \CC[x_0,\dots,x_\ell]$ and $F_2\in \CC[x_{\ell+1},\dots,x_{n+1}]$; 
  \end{enumerate}
  then $p^r$ is the order of an automorphism of a quasi-smooth hypersurface of $\PP_a^{n+1}$
\end{proposition}

\begin{proof}
To prove the proposition, it is enough to provide a quasi-smooth hypersurface of dimension $n$ and degree $d$ in $\PP^{n+1}_a$ admitting an automorphism of order $p^r$. By $(i)$ in the proposition, up to reordering the set of weights $a=(a_0,a_1,\ldots,a_{n+1})$, there exists an index $1\leq \ell\leq n+1$ such that $a_\ell m_\ell+a_0=d$ and $a_{i}m_{i}+a_{i+1}=d$, for all $i\in \{0,1,\ldots, \ell-1\}$, where every $m_{i}\in \ZZ_{>0}$. This ensures that the polynomial
\begin{align*}
\displaystyle{F'_1=\sum_{j=0}^{\ell-1}x_{j}^{m_{j}}x_{j+1}+x_{\ell}^{m_{\ell}}x_{0}},
\end{align*}
is homogeneous of degree $d$ in $\CC[x_0,\dots,x_{\ell}]$ with weights $(a_0,a_1,\dots,a_\ell)$. Moreover, we will show in \cref{Klein} below that the hypersurface defined by $F'_1$ in $\PP^{\ell}_a$ is quasi-smooth. By $(iii)$ in the proposition, we know that there exists $F_1$ and $F_2$ such that $F_1+F_2$ defines a quasi-smooth hypersurface in $\PP^{n+1}_a$. The Jacobian Criterion implies that $F_1+F_2$ defines a quasi-smooth hypersurface if and only if $F_1$ and $F_2$ also define quasi-smooth hypersurfaces. Hence, we can assume without loss of generality that $F_1=F'_1$.
 
Finally, by $(ii)$ in the proposition, we have that
\begin{align*}
     \displaystyle{(-1)^{\ell+1}\prod_{t=0}^{\ell} \left(\frac{d-a_t}{a_t}\right)\equiv 1\pmod{p^r}}\,.
\end{align*}
By this last relation, we have that $F=F_1+F_2$ is invariant by the diagonal automorphism with signature
\begin{align*}
    \sigma=\left(1,-\frac{d-a_1}{a_0},\frac{(d-a_2)(d-a_1)}{a_1a_0},\ldots,(-1)^{\ell}\ \prod_{t=1}^{\ell}\left(\frac{d-a_t}{a_{t-1}}\right) ,\underbrace{0,\ldots,0}_{(n+2-\ell)\text{- times}} \right)\in (\ZZ/p^r\ZZ)^{n+2}\,.
\end{align*}
Hence, the corresponding quasi-smooth hypersurface $X=V(F)$ in $\PP^{n+1}_a$ is invariant under the automorphism $\varphi$ of order $p^r$. This concludes the proof.
\end{proof}

The most frequently encountered case of hypersurfaces in weighted projective space in geometric applications is when each $a_i$ divides $d$. In this case, we provide in \cref{order_weights_divides_d} a complete criterion for determining the prime numbers that appear as the order of an automorphism of a smooth hypersurface of degree $d$ in $\PP_a^{n+1}$.

\cref{order_weights_divides_d} includes the classical projective space, where each $a_i = 1$, which was proven in \cite[Proposition~2.2]{GL13} and can be regarded as its generalization. Moreover, in \cite{Sano_Tasin}, the authors study the $K$-stability of quasi-smooth weighted Fano hypersurfaces $X \in \PP_{a}^{n+1}$ of degree $d$ and characterize their finite automorphism groups under the assumption that each weight $a_i$ divides $d$; our \cref{order_weights_divides_d} provides a complete criterion for the prime orders that can occur in this setting. In addition, the classification of Fano threefolds containing a smooth rational surface with ample normal bundle, as presented in \cite{Campana_Flenner}, also falls under the hypotheses of \cref{order_weights_divides_d}, as we show below in \cref{ex:campana-flenner}.

\begin{theorem}\label{order_weights_divides_d}
     Let $n, d, p$ be positive integers, with $p$ prime and $d\geq 3$. Let $a \in \ZZ_{>0}^{n+2}$ and assume that the conditions of \cref{weighted-matsumura-monsk}~$(i)$ hold. If $a_i$ divides $d$ for all $i\in\{0,1,\ldots,n+1\}$, then $p$ is the order of an automorphism of a well formed quasi-smooth hypersurface $X$ of $\PP_a^{n+1}$ if and only if one of the following conditions hold:
  \begin{enumerate}
    \item[$(a)$] $p$ divides $d$; or 
    \item[$(b)$] $a_ip$ divides $d-a_j$, for some $i,j\in\{0,1,\ldots,n+1\}$ with $i\neq j$; or
  \item[$(c)$]  after possibly reordering the weights $a = (a_0,a_1, \ldots, a_{n+1})$, there exists $\ell$ with $1 \leq \ell \leq n+1$ such that $a_0=a_1=\cdots=a_\ell$ and 
  $$
  \left(1-\frac{d}{a_0}\right)^{\ell+1} \equiv 1\pmod{p}\,.
  $$
\end{enumerate}
 \end{theorem}

\begin{proof}
To prove the ``only if" part of the theorem, assume that $X$ is a quasi-smooth hypersurface that admits an automorphism $\varphi$ of order $p$ prime. If $p$ divides $d$, then we are in case $(a)$. If there exists $i,j$ with $i\neq j$ such that $a_ip$ divides $d-a_j$, then we are in case $(b)$. 

Assume now that neither of these two conditions happen. In particular, $p$ does not divide $d$, and $a_ip$ does not divide $d - a_j$ for any $i \neq j$, for otherwise we would be in case $(a)$ or $(b)$ respectively. We are 
therefore in the setup of \cref{order_automorphism_weight}. By \cref{order_automorphism_weight}, there exists an index $\ell \in\{1,2,\dots,n+1\}$ such that the following conditions hold:
\begin{itemize}
    \item [$(i)$] $d \equiv a_0 \pmod{a_{\ell}}$, and $d \equiv a_{i+1}\pmod{a_i}$ for all $0 \leq i < \ell$; and
    \item[$(ii)$] $\displaystyle{(-1)^{\ell+1}\prod_{t=0}^{\ell} \left(\frac{d-a_t}{a_t}\right) \equiv 1 \pmod{p}}$.
\end{itemize}
By $(i)$ and the hypothesis of the theorem, we have that exists $k_i, m_i\in\ZZ_{>0}$ such that
\begin{align*}
    a_\ell m_\ell +a_{0}=d=k_\ell a_\ell\quad\mbox{and}\quad a_im_i+a_{i+1}=d=k_ia_i,\quad \mbox{for all }0 \leq i < \ell  \,.
\end{align*}
We conclude that $a_\ell$ divides $a_0$ and $a_i$ divides $a_{i+1}$ for all $0\leq i<\ell$. This yields $a_0=a_1=\dots=a_{\ell}$. And now $(c)$ follows directly from $(ii)$.

\medskip

To prove the ``if'' part of the theorem, it is enough to provide a quasi-smooth hypersurface $X=V(F)$ of dimension $n$ and degree $d$ in $\PP^{n+1}_a$ admitting an automorphism of order $p$ in each of the cases.

Assume first that $(a)$ is fulfilled, i.e., $p$ divides $d$. Letting  $m_k=\tfrac{d}{a_k}$ for all $k\in \{0,1,\dots,n+1\}$ we let $X=V(F)$, where 
\begin{align*} 
    F=x_0^{m_0}+x_1^{m_1}\dots+x_{n}^{m_{n}}+ x_{n+1}^{m_{n+1}}\,.
\end{align*}
A direct computation shows that $X$ is quasi-smooth. Since $\gcd(a_0,a_1,\ldots,a_{n+1})=1$, we have that $\gcd(a_i,p)=1$ for some $i\in \{0,1,\ldots,n+1\}$. Since $p$ divides $d$ and $\gcd(a_i,p)=1$ we have that $p$ divides $m_i$. This yields that $X$ admits the automorphism of order $p$ whose signature is
\begin{align*}
    \sigma=(1,1,\ldots,1, \underbrace{\xi_{p} }_{i\text{-th place}},1,\ldots,1)
\end{align*}
where $\xi_p$ is a primitive $p$-th root of unity, proving the theorem in this case.

Assume now that $(b)$ is fulfilled, i.e., $a_ip$ divides $d-a_j$, for some $i,j\in\{0,1,\ldots,n+1\}$ with $i\neq j$. Letting $m_k=\tfrac{d}{a_k}$ for all $k\neq i$ and $m_i=\tfrac{d-a_j}{a_i}$, we let $X=V(F)$, where 
$$F=x_{i}^{m_i}x_j+\sum_{k\not=i}x_k^{m_k}\,.$$
A direct computation shows that $X$ is quasi-smooth. Since $a_ip$ divides $d-a_j$, we have that $p$ divides $m_i$. This yields that $X$ admits the automorphism of order $p$ whose signature is
\begin{align*}
    \sigma=(1,1,\ldots,1, \underbrace{\xi_{p} }_{i\text{-th place}},1,\ldots,1)
\end{align*}
where $\xi_p$ is a primitive $p$-th root of unity, proving the theorem in this case.

Finally, assume that $(c)$ is fulfilled. In this case, \cref{order_automorphism}~$(iii)$ holds for every reordering of the weights and every $\ell$. Now, the existence of a quasi-smooth hypersurface $X=V(F)$ in $\PP^{n+1}_a$ admitting an automorphism of order $p$ is guaranteed by \cref{order_automorphism}. Moreover, the explicit polynomial $F$ defining $X$ is provided in the proof of \cref{order_automorphism}. This proves the theorem in this case and concludes the proof.
\end{proof}

\begin{remark} \label{rem:weighted-primo-chico}
    In \cref{order_weights_divides_d}, if the weighted projective space $\PP_a^{n+1}$ is different from the usual projective space, i.e., if $a\neq (1,1,\dots,1)$, then $\ell$ in $(c)$ has to be less or equal than $n-1$ since otherwise $\PP_a^{n+1}$ is not well-formed.
\end{remark}

\begin{corollary}\label{cota}
Let $n,d,p$ be positive integers with $p$ prime and $d\geq 3$. Let $a\in\ZZ_{>0}^{n+2}$ be such that $a_i$ divides $d$ for all $i\in\{0,1,\ldots,n+1\}$. Assume that the conditions of \cref{weighted-matsumura-monsk}~$(i)$ hold. If $p$ is the order of an automorphism of a quasi-smooth hypersurface $X$ of dimension $n$ and degree $d$ in $\PP_a^{n+1}$, then
$$p\leq \max\left\{d,\left(\frac{d}{a_i}-1\right)^{n_i-1}, \mbox{ for all } i\in\{0,1,\dots,n+1\}\right\} \,,$$
where $n_i$ be the number of times that the weight $a_i$ appears in $a$.
\end{corollary}

\begin{proof}
We will prove the corollary by contradiction. Assume that
\begin{align*}
    p>\max\left\{d,\left(\frac{d}{a_i}-1\right)^{n_i-1}, \mbox{ for all } i\in\{0,1,\dots,n+1\}\right\}\,.
\end{align*}
Hence, 
\begin{align}\label{p_mayor_k}
     p>d,\quad \mbox{and}\quad p>\left(\frac{d}{a_i}-1\right)^{n_i-1},  \mbox{ for all } i\in\{0,1,\dots,n+1\}\,
 \end{align}

  Since $p$ is the order of an automorphism of a quasi-smooth hypersurface, by \cref{order_weights_divides_d}, we have that one of the following conditions hold:
  \begin{enumerate}
    \item[$(a)$] $p$ divides $d$; or 
    \item[$(b)$] $a_ip$ divides $d-a_j$, for some $i,j\in\{0,1,\ldots,n+1\}$ with $i\neq j$; or
  \item[$(c)$]  after possibly reordering the weights $a = (a_0,a_1, \ldots, a_{n+1})$, there exists $\ell\in\{1,2,\ldots,n+1\}$ such that $a_0=a_1=\cdots=a_\ell$ and 
  $$
  \left(1-\frac{d}{a_0}\right)^{\ell+1} \equiv 1 \pmod{p}\,.
  $$
\end{enumerate}
\medskip
If $(a)$ or $(b)$ hold, then $p\leq d$ which yields a contradiction. Assume now that $(c)$ holds. Since $a_0=a_1=\cdots=a_\ell$, we have that $\ell+1\leq n_0 $. Then
\begin{align}\label{formula_p}
 \nonumber \left(1-\frac{d}{a_0}\right)^{\ell+1}-1&=kp,\quad \text{for some }k\in\ZZ, \text{ and so} \\
 \left(\frac{d}{a_0}-1\right)^{\ell+1}+(-1)^\ell&=kp,\quad \text{for some }k\in\ZZ_{>0}.
  \end{align}
 
 By \eqref{p_mayor_k}, we have that $p>\left(\tfrac{d}{a_0}-1\right)^{n_0-1}$. This yields $\ell+1=n_0-1$ or $\ell+1=n_0$. Assume first that $\ell+1=n_0-1$. Then by \eqref{formula_p}, we have that this is only possible if $k=1$ yielding
$$ \left(\frac{d}{a_0}-1\right)^{n_0-1}+(-1)^{n_0-2}=p\,.$$
Assume now that $\ell+1=n_0$. Then by \eqref{formula_p}, we have that this is only possible if 
$$ \left(\frac{d}{a_0}-1\right)^{n_0}+(-1)^{n_0-1}=kp,\quad \text{for some }k\in\left\{1,2,\dots,\tfrac{d}{a_0}-1\right\}.$$
In both cases, we conclude by the binomial expansion that $\tfrac{d}{a_0}\geq 2$ divides $p$ providing a contradiction. 
\end{proof}

As  an application of \cref{order_weights_divides_d}, we provide the following example.

\begin{example} \label{ex:campana-flenner}
For simplicity, in this example we denote the weighted projective space $\PP^{n+1}_a$ simply by $\PP(a)$. The classification of Fano threefolds containing a smooth rational surface with ample normal bundle, as presented in~\cite{Campana_Flenner}, includes four families of hypersurfaces: a smooth cubic hypersurface in $\PP^4=\PP(1,1,1,1,1)$, a quartic hypersurface in $\PP(1,1,1,1,2)$, a sextic hypersurface in $\PP(1,1,1,2,3)$, and a sextic hypersurface in $\PP(1,1,2,2,3)$; see also~\cite{prokhorov}.

It follows from \cref{order_weights_divides_d} that smooth cubic hypersurfaces in $\PP^4$ may admit automorphisms of order $p = 2$, $3$, $5$ and $ 11$; quartic hypersurface in $\PP(1,1,1,1,2)$ may have automorphisms of order $p = 2$, $3$, $5$ and $7$; sextic hypersurface in $\PP(1,1,1,2,3)$ may admit automorphisms of order $p = 2$, $3$, $5$ and $7$; and sextic hypersurface in $\PP(1,1,2,2,3)$ may have automorphisms of order $p = 2$, $3$ and $5$. We provide the computation of the case of the sextic hypersurface in $\PP(1,1,1,2,3)$ as an example of the computations. By \cref{order_weights_divides_d}~$(a)$ and $(b)$, there are quasi-smooth hypersurfaces with automorphism of orders $2$, $3$ and $5$. Now, to apply \cref{order_weights_divides_d}~$(iii)$, we have that $\ell=1$ or $\ell=2$. Then the possible prime orders different from $2$, $3$ and $5$ satisfy
\begin{align*}
    (1-6)^2=2^3\cdot 3+1\equiv 1 \pmod{p},\quad\mbox{and}\quad
    (1-6)^{3}=-2\cdot3^2\cdot7+1\equiv1\pmod{p}.
\end{align*}
Then, the sextic hypersurface in $\PP(1,1,1,2,3)$ admits automorphisms of prime order $p=2$, $3$, $5$ and $7$.
\end{example}

\section{Automorphism of maximal prime order and weighted Klein hypersurfaces} \label{sec3}

In \cref{rem:weighted-primo-chico}, we showed that the maximal prime numbers that can appear as the order of an automorphism of a quasi-smooth hypersurface in a weighted projective space are relatively small compared to the classical case of usual projective space, under the assumption that every weight $a_i$ divides $d$. In this section, we investigate the opposite case, where every weight $a_i$ is relatively prime to $d$. 

In this context, similarly to the situation in~\cite{GL13}, the natural generalization of Klein hypersurfaces arise as the varieties admitting automorphisms of the largest possible prime order. We begin by proving the following corollary, which provides a bound in this setting.

\begin{corollary} \label{cota-klein}
Let $n, d, p$ be positive integers, with $p > d$ a prime number and $d\geq 3$. Let $a \in \ZZ_{>0}^{n+2}$, and assume that $\gcd(a_i, d) = 1$ for all $i \in \{0,1,\dots,n+1\}$. Further, assume that the condition in \cref{weighted-matsumura-monsk}~$(i)$ hold. If $p$ is the order of an automorphism of a quasi-smooth hypersurface of dimension $n$ and degree $d$ in $\PP^{n+1}_a$, then
$$
p < \left( \frac{\max(a)}{d - \max(a)} \right) \prod_{t=0}^{n+1} \left( \frac{d - a_t}{a_t} \right)\,,
$$
where $\max(a)$ denotes the maximum of the weights $a_0, a_1,\ldots, a_{n+1}$.
\end{corollary}

\begin{proof}
    We prove the statement by contradiction. Assume that $p>\left(\frac{\max(a)}{d-\max(a)}\right)\prod_{t=0}^{n+1}\left(\frac{d-a_t}{a_t}\right)$. Up to reordering the weight, we may and will assume that $\max(a)=a_{n+1}$ so that
\begin{align}\label{eq:cota-p}
    p>\left(\frac{\max(a)}{d-\max(a)}\right)\prod_{t=0}^{n+1}\left(\frac{d-a_t}{a_t}\right)=&\prod_{t=0}^{n}\left(\frac{d-a_t}{a_t}\right).
\end{align}
    Since $p>d$ is the order of an automorphism of a quasi-smooth hypersurface,  by \cref{order_automorphism_weight}, we have that there exists an index $\ell\in\{1,2,\dots,n+1\}$ such that the following conditions hold:
  \begin{enumerate}
      \item[$(i)$] $d \equiv a_0 \pmod{a_{\ell}}$, and $d \equiv a_{i+1} \pmod{a_i}$ for $0 \leq i < \ell$; and
      \item[$(ii)$] $\displaystyle{(-1)^{\ell+1}\prod_{t=0}^{\ell} \left(\frac{d-a_t}{a_t}\right)\equiv 1 \pmod{p}}$.
  \end{enumerate}
Now, $(ii)$ yields 
\begin{align} \label{eq:la-milagrosa}
\prod_{t=0}^\ell \left(\frac{d}{a_t}-1\right) +(-1)^\ell=kp,\quad \mbox{for some } k\in \ZZ_{>0}\,.
\end{align}
By \eqref{eq:cota-p}, we have that $\ell=n$ or $\ell={n+1}$. Assume first that $\ell=n$. Then, by \eqref{eq:la-milagrosa}, we have that this is only possible if $k=1$ yielding
$$\prod_{t=0}^n \left(\frac{d}{a_t}-1\right) +(-1)^n=p\,.$$
Since $\gcd(a_i,d)=1$ for every $i$, we can take this last equality modulo $d$ to conclude that $d$ divides $p$ which is a contradiction.

Assume now that $\ell=n+1$. Then by \eqref{eq:la-milagrosa}, we have that this is only possible if
$$\prod_{t=0}^{n+1} \left(\frac{d}{a_t}-1\right) +(-1)^{n+1}=kp,\quad \mbox{for some } k\in \left\{1,2,\dots, \tfrac{d}{a_{n+1}}-1\right\}\,.$$
Again, since $\gcd(a_i,d)=1$ for every $i$, we can take this last equality modulo $d$ to conclude that $d$ divides $kp$. Since $k<d$, we conclude that some divisor $d_0>1$ of $d$ divides $p$ which is again a contradiction.
\end{proof}

We now introduce the natural notion of Klein hypersurface in the context of weighted projective spaces.

\begin{definition}
Let $n, d$ be positive integers, and let $a \in \ZZ_{>0}^{n+2}$. Let $X$ be a quasi-smooth hypersurface of degree $d$ in $\PP^{n+1}_a$. We say that $X$ is a weighted Klein hypersurface in $\PP^{n+1}_a$ if $X$ is isomorphic to  $V(K)$, where $K$ is a homogeneous polynomial of the form
\begin{align*}
K = x_0^{m_0}x_1 + x_1^{m_1}x_2 + \ldots + x_n^{m_n}x_{n+1} + x_{n+1}^{m_{n+1}}x_0\,.
\end{align*}
\end{definition}

\begin{remark}
The weighted Klein hypersurface does not exist for every choice of $d$ and $a$. Indeed, a straightforward verification shows that there is no weighted Klein hypersurface of degree $4$ in $\PP^{3}_a$ when $a = (1,1,1,2)$.

These hypersurfaces were previously studied in the final section of \cite{kollar}. More recently, in \cite{urzua}, Urzúa and Yáñez focused on the surface case in detail, referring to them as Kollár hypersurfaces. We adopt the name weighted Klein hypersurface following the analogy with the classical Klein hypersurface studied in \cite{GL13}.
\end{remark}

In the following proposition we show that a weighted Klein Hypersurface is almost always quasi-smooth. See also \cite[Example~3.5]{GL13}.

\begin{proposition}\label{Klein}
Let $n,d$ be positive integers with $d\geq 2$. Let $a\in\ZZ_{>0}^{n+2}$. Assume that $X=V(K)$ is a weighted Klein Hypersurface in $\PP^{n+1}_a$. Then $X$ is not quasi-smooth if and only if $a=(1,1,\dots,1)$, $d=2$, and $n\equiv 2\pmod{4}$.
\end{proposition}
\begin{proof}
    Assume that $\alpha=[\alpha_0:\alpha_1:\dots:\alpha_{n+1}]\in \PP_a^{n+1}$ is a singular point of $X$. Then 
    \begin{align*}
        K(\alpha)=0\quad \text{and}\quad \frac{\partial K}{\partial x_{i}}(\alpha)=0,\quad \text{for all } i\in\{0,1,\dots,n+1\}.
    \end{align*}
    This yields 
\begin{align*}
\frac{\partial K}{\partial x_{0}}(\alpha)&=\alpha_{n+1}^{m_{n+1}}+m_{0}\alpha_{0}^{m_{0}-1}\alpha_{1}=0, \\
\frac{\partial K}{\partial x_{i}}(\alpha)&=\alpha_{i-1}^{m_{i-1}}+m_{i}\alpha_{i}^{m_{i}-1}\alpha_{i+1}=0, \text{ for all } i\in\{1,2,\dots,n\}, \mbox{ and}\\
\frac{\partial K}{\partial x_{n+1}}(\alpha)&=\alpha_{n}^{m_{n}}+m_{n+1}\alpha_{n+1}^{m_{n+1}-1}\alpha_{0}=0\,.
\end{align*} 
Remark that if any $\alpha_j=0$, the identities above yield that $\alpha_i=0$, for all $i\in\{0,1,\dots,n+1\}$. Hence, we may and will assume that $\alpha_i\neq 0$ , for all $i\in\{0,1,\dots,n+1\}$. Multiplying each of these identities by $\alpha_i$, we obtain
\begin{align*}
\alpha_{n+1}^{m_{n+1}}\alpha_{0}&=-m_{0}\alpha_{0}^{m_{0}}\alpha_{1},\\
\alpha_{i-1}^{m_{i-1}}\alpha_{i}&=-m_{i}\alpha_{i}^{m_{i}}\alpha_{i+1}, \text{ for all } i\in\{1,2,\dots,n\}, \mbox{ and} \\
\alpha_{n}^{m_{n}}\alpha_{n+1}&=-m_{n+1}\alpha_{n+1}^{m_{n+1}}\alpha_{0}\,.
\end{align*}
Then, we obtain that
\begin{align*}
\alpha_{i}^{m_{i}}\alpha_{i+1}=(-1)^{n+1-i}\alpha_{n+1}\alpha_{0}\left(\prod_{j=i+1}^{n+1}m_{j}\right), \quad\text{for all } i\in \{0,1,\dots,n\}.
\end{align*}
Replacing these identities in $K(\alpha)$, we obtain
\begin{align*}
    K(\alpha)=R\cdot\alpha_{n+1}\alpha_{0} \quad \text{where}\quad R=1+\sum_{i=1}^{n+1}\left\{(-1)^{n-i}\left( \prod_{j=i}^{n+1}m_j\right)\right\}\,.
\end{align*}
Now, in this last equality we have $R=0$ if and only if $a=(1,1,\dots,1)$, $d=2$, and $n$ is even. If $a=(1,1,\dots,1)$ and $d=2$ then $K$ is a quadratic polynomial. A routine computation shows that in this case $X=V(K)$ is not smooth if and only if $n\equiv 2\pmod{4}$ proving the proposition.
\end{proof}
\begin{lemma}\label{p_maximum}
    Let $n,d,p$ be positive integers with $p>d$ prime. Let $a\in\ZZ_{>0}^{n+2}$ and assume that $\gcd(a_i,d)=1$ for all $i\in \{0,1,\dots,n+1\}$. Assume further that the conditions in \cref{weighted-matsumura-monsk}~$(i)$ are fulfilled. If $p$ is the order of an automorphism of a quasi-smooth hypersurface of dimension $n$ and degree $d$ in $\PP^{n+1}_a$. Then 
    $$p = \frac{1}{d}\ \left\{\ \prod_{t=0}^{n+1}\left(\frac{d-a_t}{a_t}\right)+(-1)^{n+1}\ \right\}\,$$
    is the largest prime that can be the order of an automorphism of a quasi-smooth hypersurface of $\PP_{a}^{n+1}$.
\end{lemma}
\begin{proof}
Assume that $p>d$ is the order of an automorphism of a quasi-smooth hypersurface of dimension $n$ and degree $d$ in $\PP^{n+1}_a$. Then, by \cref{order_automorphism_weight}~$(ii)$, we have that 
$$A:=(-1)^{\ell+1}\prod_{t=0}^{\ell} \left(\frac{d-a_t}{a_t}\right) - 1 \equiv 0 \pmod{p}\,,$$
and by \cref{order_automorphism_weight}~$(i)$ $A$ is an integer. In particular, this implies $p\leq |A|$. Moreover, the maximum possible value of $|A|$ is attained when $\ell=n+1$. Consequently,
$$p\leq \prod_{t=0}^{n+1} \left(\frac{d-a_t}{a_t}\right) + (-1)^{n+1}$$
but under the hypothesis that $\gcd(a_i,d)=1$ we have that this number is divisible by $d$. Indeed, denoting the inverse of $a_i$ modulo $d$ by $[a_i]^{-1}$, we have that
\begin{align*}
    \prod_{t=0}^{n+1} \left(\frac{d-a_t}{a_t}\right) + (-1)^{n+1}&\equiv\prod_{t=0}^{n+1}  [a_t]^{-1}\left(d-a_t\right)+(-1)^{n+1} \pmod{d}\\
    &\equiv  (-1)^{n+2}+(-1)^{n+1} \pmod{d}\\
    &\equiv 0 \pmod{d}
\end{align*}
Hence, we conclude that the largest prime that can be the order of an automorphism of a quasi-smooth hypersurface of $\PP^{n+1}_a$ is
$$p = \frac{1}{d}\ \left\{\ \prod_{t=0}^{n+1}\left(\frac{d-a_t}{a_t}\right)+(-1)^{n+1}\ \right\}\,.$$
\end{proof}

In this context, we will prove the following theorem that is the main result of this section. 

\begin{theorem}\label{th:Klein}
Let $n,d$ be positive integers with $d\geq 3$. Let $a\in\ZZ_{>0}^{n+2}$. Assume further that the conditions in \cref{weighted-matsumura-monsk}~$(i)$ are fulfilled. Let 
$$p=\frac{1}{d}\ \left\{\ \prod_{t=0}^{n+1}\left(\frac{d-a_t}{a_t}\right)+(-1)^{n+1}\ \right\}\,,$$
and assume that $p>d$ and that $p$ is prime. 
If a quasi-smooth hypersurface $X=V(F)$ of dimension $n$ and degree $d$ admits an automorphism $\varphi$ of order $p$ then $X$ is isomorphic to the weighted Klein hypersurface.
\end{theorem}

Before proving \cref{th:Klein}, we need some technical lemmas. Our strategy follows that of \cite[Theorem~3.7]{GL13}: we compute a basis for the eigenspace of the eigenvalue $1$ of the action of 
$\widetilde{\varphi}^*$ on $S$, and use it to pin down the defining 
polynomial $F$ up to isomorphism.

\begin{lemma} \label{base-like-decomposition}
    Let $q \in \ZZ$ be an integer and let $m_i\in \ZZ_{\geq 2}$, $i \in \{1,2,\dots n\}$ be a finite collection of integers greater or equal than two. Then, there exist unique non-negative integers $b_i\in \ZZ_{\geq 0}$, $i \in \{0,1,\dots n\}$ with $b_i < m_{i+1}$ for all $i\in\{0,1,\ldots,n-1\}$ such that 
    \begin{align*}  q=b_0+\sum_{i=1}^{n}\left((-1)^{i}b_{i}\prod_{k=1}^{i}m_k\right),
    \end{align*}
\end{lemma}

\begin{proof}
The proof of this lemma follows from the Euclidean division. We proceed by induction on $n$. If $n=1$, then there is only one $m_1$ and by the Euclidean division, there exists unique $c$ and $r$ with $0\leq r<m_1$ such that 
$$q = r+cm_1\,.$$
Hence, $b_0$ is the remainder of the Euclidean division $r$ and $b_1$ is $-c$, the negative of the quotient of the same Euclidean division.

Assume now that the lemma holds for all  $n'<n$ and we will prove it for $n$. Let $m=\prod_{k=1}^{n}m_k$. The Euclidean division assures that there exists unique 
$c$ and $r$ with $0\leq r< m$ such that 
$$q = r+cm= r+c\prod_{k=1}^{n}m_k$$
Now, let $b_n=(-1)^n c$ so that 
$$q = r+(-1)^nb_n\prod_{k=1}^{n}m_k\,.$$
Now the lemma follows applying the induction hypothesis to $r<\prod_{k=1}^{n}m_k$ with the collection of $n'=n-1$ integers $m_i\in \ZZ_{\geq 2}$, $i \in \{1,2,\dots n-1\}$.
\end{proof}

\begin{remark}\label{otra_forma}
    We will apply \cref{base-like-decomposition} by taking
    \begin{align*}
        m_i = \frac{d - a_{i+1}}{a_i}, \quad \text{for } i \in \{0, 1, \dots, n\}\,.
    \end{align*}
    With this notation, the signature of an automorphism $\varphi$ of order $p$ 
    of a quasi-smooth hypersurface $X \subset \PP^{n+1}_{a}$ of degree $d$ and 
    dimension $n$, which is given by
    \begin{align*}
        \sigma = \left(1,\
        -\frac{d - a_1}{a_0},\
        \frac{(d - a_2)(d - a_1)}{a_1 a_0},\
        \ldots,\
        (-1)^{n+1} \prod_{t=1}^{n+1} \frac{d - a_t}{a_{t-1}}
        \right) \in (\ZZ/p\ZZ)^{n+2},
    \end{align*}
    can be expressed compactly as
    \begin{align*}
        \sigma = \left(1,\ -m_0,\ m_0 m_1,\ \ldots,\ 
        (-1)^{n+1} \prod_{t=0}^{n} m_t \right).
    \end{align*}
\end{remark}

The following lemma is the last ingredient in the proof of \cref{th:Klein}.

\begin{lemma}\label{lema_base}
Let $n,d,p$ be positive integers with $p$ prime. Let $a\in\ZZ_{>0}^{n+2}$ and assume $d>2a_i$, for all $i\in\{0,1,\ldots,n+1\}$. Let $b_i\in \ZZ_{\geq 0}$, $i\in\{0,1,\dots,n+1\}$ with $b_i<m_{i}$ for all $i\in\{0,1,\dots,n\}$. Then 
$$b_0+\sum_{k=1}^{n+1}b_k\sigma_k=0\quad \mbox{if and only if}\quad  b_i=0,\ \mbox{for all } i\in\{0,1,\ldots,n+1\}\,.$$
\end{lemma}
\begin{proof}
Since $a_i\geq 1$, the condition $d>2a_i$ implies that $d-a_{i}> \max{a}$, for all $i\in\{0,1,\dots,n+1\}$. Hence, $m_i>1$ for all $i\in\{0,1,\dots,n \}$, and the lemma follows directly from \cref{base-like-decomposition} and \cref{otra_forma}.
\end{proof}

\begin{lemma}\label{lema_unico_saturado}
Let $n,d$ be positive integers with $d>2\max(a)$, and let $a\in\ZZ_{>0}^{n+2}$.
Let $\mathbf{x}^r = x_0^{r_0}\cdots x_{n+1}^{r_{n+1}}$ be a monomial of degree
$d$ with respect to the grading given by $a$, so that
$\sum_{i=0}^{n+1}a_ir_i=d$ and $r_i\in\ZZ_{\geq 0}$. Then:
\begin{itemize}
\item[$(i)$] $r_i \leq m_i$ for all $i\in\{0,1,\dots,n\}$.
\item[$(ii)$] There is \emph{at most one} index $j\in\{0,\dots,n\}$ with $r_j = m_j$.
\item [$(iii)$] If $r_j=m_j$ for some $j\in\{0,\dots,n\}$, then $r_i<m_i$ for all
      $i\in\{0,\dots,n\}$ with $i\neq j$.
\end{itemize}
\end{lemma}

\begin{proof}
To prove $(i)$, we will prove by contradiction. Assume that $r_i>m_i$ for some $i\in\{0,1,\dots,n\}$. Since $m_i> 1$, then
\begin{align*}
    a_ir_i+a_{i+1}r_{i+1}&> a_im_i+a_{i+1}
\end{align*}
which give us a contradiction since $d\geq a_ir_i+a_{i+1}r_{j+1}$ and $a_im_i+a_{i+1}=d$. Hence $r_i\leq m_i$ for all $i\in\{0,1,\dots,n\}$.

To prove the last two statements, suppose that $r_j = m_j$ for some $j \in \{0,1,\dots,n\}$. Then
\begin{align*}
    a_j m_j = d - a_{j+1},
\end{align*}
and therefore
\begin{align*}
 \sum_{i \neq j} a_i r_i = d - a_j m_j = a_{j+1}.   
\end{align*}
Now, let $i \in \{0,1,\dots,n\}$ with $i \neq j$. Then
\begin{align*}
    a_i r_i \leq a_{j+1} < d - a_{i+1} = a_i m_i,
\end{align*}
where the strict inequality follows from the assumption $d > 2\max(a)$. Consequently, $r_i < m_i$ for all $i \neq j$, which proves both $(iii)$ and $(ii)$.
\end{proof}

Having established these results, we now proceed to the proof of \cref{th:Klein}.
   
\begin{proof}[Proof of \cref{th:Klein}]
Assume that $X=V(F)$ is a quasi-smooth hypersurface of dimension $n$ and degree $d$, where $F$ is a homogeneous polynomial of degree $d$ with respect to the grading given by $a\in\ZZ_{>0}^{n+2}$. Assume further that $X$ admits an automorphism $\varphi$ of order $p>d$ prime, where  
\begin{align} \label{eq:dp}
p=\frac{1}{d}\ \left\{\ \prod_{t=0}^{n+1}\left(\frac{d-a_t}{a_t}\right)+(-1)^{n+1}\ \right\}\,.
\end{align}
Then, by \cref{order_automorphism_weight} and \cref{p_maximum}, we have $\ell=n+1$. We now fix an order of the weights $a$ so that we have $d=a_{n+1}m_{n+1}+a_0$ and $d=a_im_i+a_{i+1}$, for all $i\in\{0,1,\ldots,n\}$. And by \cref{liftable}, with the chosen order of the weight, the automorphism $\varphi$  of order prime $p$ is given by the following signature
\begin{align}\label{eq:signatura-automorphism}
\sigma=\left(\sigma_0,\sigma_1,\dots,\sigma_{n+1} \right)\in (\ZZ/p\ZZ)^{n+2}\, ,
\end{align}
where
$$\sigma_0=1, \quad\mbox{and}\quad \sigma_i=(-1)^i\prod_{t=1}^i\frac{d-a_t}{a_{t-1}},\quad\mbox{for all } i\in\{1,2,\dots,n+1\}\, ,$$
Remark that
\begin{align}\label{eq:sigma_i}
    \sigma_{i+1}=-\left(\frac{d-a_{i+1}}{a_i}\right)\sigma_i=-m_{i}\sigma_i,\quad\mbox{for all } i\in\{0,1,\dots,n\}.
\end{align}

Let $S$ be the subspace of $\CC[x_0,\dots,x_{n+1}]$ composed of homogeneous polynomials of degree $d$ with respect to the grading given by $a\in\ZZ_{>0}^{n+2}$. We also let $\tvarphi^*\colon S\to S$ be action of $\tvarphi^*$ on $S$. Let $\mathcal{E}\subset S$ be the eigenspace associated to the eigenvalue $1$ of the linear automorphism $\tvarphi^*$. Since $\varphi$ is an automorphism of $X$, we have that $F\in \mathcal{E}$. We will now compute a basis for $\mathcal{E}$.

Let now $\mathbf{x}^r$ be a monomial in $S$, then
\begin{align}\label{eq:in-S}
\mathbf{x}^r=x_{0}^{r_0}\cdots x_{n+1}^{r_{n+1}},\quad \mbox{with}\quad \sum_{i=0}^{n+1}a_i r_i=d,\text{ and } r_i\in \ZZ_{\geq 0}.
\end{align}
Assume further that $\mathbf{x}^r\in\mathcal{E}$. Hence, by \eqref{eq:signatura-automorphism}, we have 
\begin{align}\label{eq:in-E}
\sum_{i=0}^{n+1}\sigma_i r_i \equiv 0 \pmod{p}\quad \text{i.e.}\quad \sum_{i=0}^{n+1}\sigma_i r_i =kp, \text{ for some }k\in \ZZ.
\end{align}
We first consider the case $k \neq 0$. Using \eqref{eq:dp}, we obtain
\begin{align*}
    \sum_{i=0}^{n+1}\sigma_i r_i
    = \frac{k}{d}\left( \prod_{t=0}^{n+1}\frac{d-a_t}{a_t} + (-1)^{n+1} \right).
\end{align*}
By the definition of $\sigma_{n+1}$, it follows that
\begin{align}\label{eq:notrearranged}
    \sum_{i=0}^{n+1}\sigma_i r_i&= \frac{k}{d}\left( (-1)^{n+1} m_{n+1}\sigma_{n+1} + (-1)^{n+1} \right) 
\end{align}
Observe that the expression 
\begin{align*}
    \frac{k}{d}\left( (-1)^{n+1} m_{n+1}\sigma_{n+1} + (-1)^{n+1} \right) \, ,
\end{align*}
is an integer if and only if $k=td$, for some $t\in \ZZ\setminus \{0\}$. Moreover, by the definition of $\sigma_{n+1}$, this integer is positive when $n$ is odd and $t\in\ZZ_{>0}$, and negative when $n$ is even and $t\in\ZZ_{<0}$. Observe that in both cases, we have $r_0\geq (-1)^{n+1}t>0$.

Rewriting \eqref{eq:notrearranged}, we get
\begin{align*}
    r_0 + \sum_{i=1}^{n}\sigma_i r_i + \sigma_{n+1} r_{n+1}
    = (-1)^{n+1}\,t m_{n+1}\,\sigma_{n+1} + (-1)^{n+1}t \, ,
\end{align*}
which can be rearranged as
\begin{align}\label{eq:rearranged}
    \left(r_0 + (-1)^n t\right)
    + \sum_{i=1}^{n} r_i\sigma_i
    + \sigma_{n+1}\left(r_{n+1} + (-1)^n t\right)
    = 0.
\end{align}

We claim that $r_i<m_i$ for all $i\in\{1,2,\dots,n\}$.
Indeed, suppose $r_j=m_j$ for some $j\in\{1,2,\dots,n\}$.
By \cref{lema_unico_saturado}~$(iii)$, $r_i<m_i$ for all $i\in\{0,1,\dots,n\}$ and $i\neq j$. By \eqref{eq:sigma_i}, we have that  $\sigma_j m_j = -\sigma_{j+1}$ and replacing in \eqref{eq:rearranged}, we obtain
\begin{align*}
    \left(r_0 + (-1)^n t\right)
    + \sum_{\substack{i=1\\
    i\not=j,j+1}}^{n}\sigma_i  r_i-\sigma_{j+1}+\sigma_{j+1}r_{j+1}
    + \left(r_{n+1} + (-1)^n t\right)\sigma_{n+1}=0\\
    \left(r_0 + (-1)^n t\right)
    + \sum_{\substack{i=1\\
    i\not=j,j+1}}^{n} \sigma_i r_i +\sigma_{j+1}(r_{j+1}-1)
    + \sigma_{n+1}\left(r_{n+1} + (-1)^n tm_{n+1}\right)=0\, .
\end{align*}
Setting $b_0 = r_0+(-1)^n t$, $b_{j+1} = r_{j+1}-1$, $b_i = r_i$ for 
$i \in \{1,2,\dots,n\}\setminus\{j,j+1\}$, and $b_{n+1} = r_{n+1}+(-1)^n t\,m_{n+1}$.
All these satisfy the strict bounds of \cref{lema_base}.
Applying \cref{lema_base}, we obtain $b_i=0$ for all $i$.
From $b_0=0$ and $b_{n+1}=0$ we get $r_0 = (-1)^{n+1}t$ and 
$r_{n+1} = (-1)^{n+1}t\,m_{n+1}$. Substituting all nonzero terms into the 
degree condition~\eqref{eq:in-S} yields
$$a_0\,r_0 + a_j m_j + a_{j+1}\cdot 1 + a_{n+1}\,r_{n+1} 
= (-1)^{n+1}t\,(a_0 + a_{n+1}m_{n+1}) + (d - a_{j+1}) + a_{j+1} = d\,,$$
where we used $a_jm_j = d - a_{j+1}$.
Since $a_0 + a_{n+1}m_{n+1} = d$, this simplifies to $(-1)^{n+1}t\cdot d + d = d$,
forcing $t=0$ and contradicting $k = td \neq 0$.

Now all coefficients in \eqref{eq:rearranged} satisfy the strict bounds of
\cref{lema_base}. Applying \cref{lema_base} gives
$$r_0 + (-1)^n t = 0,\quad
r_{n+1}+(-1)^n t\,m_{n+1}=0, \quad \text{and}\quad r_i=0\ \text{ for all }i\in\{1,\dots,n\}\,.$$
From the first equation, we obtain that $r_0 = (-1)^{n+1}t$. And substituting into the second equation yields that $r_{n+1} = r_0 m_{n+1}$. Finally, using \eqref{eq:in-S}, we obtain $r_0 = 1$, and therefore $r_{n+1} = m_{n+1}$. This corresponds to the monomial $\mathbf{x} = x_{n+1}^{m_{n+1}}x_0$.

We now consider the case $k = 0$. Then
\begin{align*}
    \sum_{i=0}^{n+1}\sigma_i r_i = 0.
\end{align*}
By \cref{lema_unico_saturado}~$(ii)$, there is at most one $j\in\{0,1,\dots,n\}$ with
$r_j=m_j$. If no such $j$ exists, all $r_i<m_i$ for $i\in\{0,1, \dots,n\}$, and
\cref{lema_base} gives that every $r_i=0$. Then $\sum a_ir_i=0$, so this subcase yields no monomials in $S$.

Suppose $r_j=m_j$ for a unique $j\in\{0,1,\dots,n\}$.
By \cref{lema_unico_saturado}~$(iii)$, $r_i<m_i$ for all $i\in \{0,1,\dots,n\}$ and $i\neq j$.
Then
\begin{align*}
    \sum_{\substack{i=0\\i\neq j,j+1}}^{n+1}\sigma_ir_i
    +\sigma_jm_j+\sigma_{j+1}r_{j+1}=0.
\end{align*}
Using the definition of $\sigma_{j+1}$ from \eqref{eq:sigma_i}, this becomes
\begin{align*}
    \sum_{\substack{i=0\\i\neq j,j+1}}^{n+1}\sigma_ir_i
    +\sigma_{j+1}(r_{j+1}-1)=0.
\end{align*}
Since $r_{j+1}<m_{j+1}$, we have 
$r_{j+1}-1 < m_{j+1}$, and $r_{j+1}-1\geq 0$.
All coefficients now satisfy the strict bounds of \cref{lema_base}.
Applying \cref{lema_base}, we obtain that $r_{j+1}=1$ and $r_i=0$, for every $i\in\{0,1,\dots,n+1\}$ and $i\not= j$. Using~\eqref{eq:in-S}, we obtain $a_jm_j+a_{j+1}=d$, and this corresponds to the monomial
$$\mathbf{x} = x_j^{m_j}x_{j+1},\quad j\in\{0,1,\dots,n\}.$$

Combining both cases, a basis for the eigenspace $\mathcal{E}$ is
$$\left\{x_{n+1}^{m_{n+1}}x_0\right\}\cup\left\{x_j^{m_j}x_{j+1}\mid j=0,1,\dots,n\right\}\,.$$

Hence,
$$
F = \lambda_0 \cdot x_0^{m_0} x_1 + \lambda_1 \cdot x_1^{m_1} x_2 + \cdots + \lambda_n \cdot x_n^{m_n} x_{n+1} + \lambda_{n+1} \cdot x_{n+1}^{m_{n+1}} x_0.
$$
Since \( X = V(F) \) is quasi-smooth, by \cref{quasi-smooth_hypersurface} , $\lambda_i \neq 0$, for all $i\in \{0,1,\dots,n+1\}$ and applying a linear change of coordinates we can put
$$
F = x_0^{m_0} x_1 + x_1^{m_1} x_2 + \cdots + x_n^{m_n} x_{n+1} + x_{n+1}^{m_{n+1}} x_0\,,
$$
which proves that $X=V(F)$ is isomorphic to the weighted Klein hypersurface.

\end{proof}

\section*{Data availability statement}

Not applicable.

\section*{Conflict of interest}

The authors declare that they have no conflict of interest.

\bibliographystyle{alpha}
\bibliography{nfolds5}
\end{document}